\documentclass{amsart}

\usepackage{comment}

\usepackage{dsfont}
\usepackage{amsmath}
\usepackage{amssymb}
\usepackage{graphicx}

\usepackage{amssymb}
\usepackage{comment}

\usepackage[colorlinks,linkcolor=blue,citecolor=blue,urlcolor=blue]{hyperref}

\usepackage{amsfonts}

\usepackage{soul}

\usepackage{mathpazo}
\usepackage{color}
\usepackage{yfonts}

\usepackage{paralist}

\usepackage{stmaryrd}

\usepackage{amsxtra}

\usepackage{mathtools}

\def\Id{{\rm Id}}

\def\clb{   \color{black}}

\let\cal\mathcal

\def \R{{\mathbb R}}

\def \Z{{\mathbb Z}}

\newcommand{\T}{{\mathbb T}}
\newcommand{\prf}{{\begin{proof}}}
\newcommand{\epf}{{\end{proof}}}

\newcommand{\Q}{{\mathbb Q}}

\newcommand{\ary}{\begin{eqnarray}}
\newcommand{\eary}{\end{eqnarray}}

\newcommand{\aryst}{\begin{eqnarray*}}
\newcommand{\earyst}{\end{eqnarray*}}

\newcommand{\enmt}{\begin{enumerate}}
\newcommand{\eenmt}{\end{enumerate}}

\newtheorem{prop}{\sc Proposition}

\newtheorem{lemma}{\sc lemma}

\newtheorem{cor}{\sc corollary}

\theoremstyle{definition}

\def\bee{\begin{equation}}
\def\eee{\end{equation}}

\newtheorem{defi}{\it Definition}

\theoremstyle{rema}

\newcommand{\pdvr}[2]
{\dfrac{\partial^{#2} #1}{\partial \theta^{#2_1} \partial r^{#2_2}}}
\newcommand{\pdvrs}[2]
{\partial^{#2} #1 /\partial \theta^{#2_1} \partial r^{#2_2}}
\newtheorem{thm}{\sc Theorem}%[section]

\newcommand\myeq{\stackrel{\mathclap{\normalfont\mbox{\small def}}}{=}}

\numberwithin{equation}{section}

\author{}
\begin{document}

\title[genericity of mode-locking]{On topological genericity of  the mode-locking phenomenon}
% Sur genercite topologique de la phénomène de mode-locking

\author{Zhiyuan Zhang}
\address{Institut de Math\'{e}matique de Jussieu \& KTH Royal Institute of Technology}
\email{zzzhangzhiyuan@gmail.com}

\date{Oct. 23, 2017}

\maketitle

\renewcommand{\abstractname}{Abstract}

\begin{abstract}
 We study the circle homeomorphisms extensions over a strictly ergodic homeomorphism. Under a very mild restriction, we show that the fibered rotation number is locally constant on an open and dense subset. In the complement of this set, we found a dense subset in which every map is conjugate to a direct product. Our result provides a generalisation of Avila-Bochi-Damanik's result on ${\rm SL}(2,\mathbb{R})-$cocycles, and  J\"ager-Wang-Zhou's result on quasi-periodically forced maps, to a broader setting.
\end{abstract}

%\tableofcontents
%\addtocontents{toc}{\protect\setcounter{tocdepth}{2}   }

\section{Introduction}

The study of circle homeomorphisms is a classical subject in dynamical systems. 
For each homeomorphism $f: \T \to \T$ with a lift  $F : \R \to \R$, the limit $\rho(F) = \lim_{n \to \infty}(F^{n}(x) - x)/n$ exists and is independent of $x$. The \textit{rotation number} of $f$ is defined as $\rho(f) = \rho(F) \mod \ 1$. It is already known to Poincar\'e that the dynamic of $f$ is largely determined by $\rho(f)$. The function $f' \mapsto \rho(f')$ is locally constant at $f$, or in other terms $f$ is \textit{mode-locked}, if and only if there exists a non-empty open interval $I \subsetneqq \T$ such that $f^p(\overline{I}) \subset I$ for some $p \in \Z$. In particular, such behaviour occurs only when $\rho(f)$ is rational.

Inspired by a question of Herman in \cite{H}, Bjerkl\"ov and J\"ager  found in \cite{BJ} a precise analogy to the above one-dimensional result for \textit{quasi-periodically forced maps} on $\T^2$ defined as follows.
\begin{defi}\label{defi qpf on t2}
A quasi-periodically forced map  on $\T^2$ (qpf-map for short) is a homeomorphism $f: \T^2 \to \T^2$, homotopic to the identity, with  skew-product structure $f(\theta,x) = (\theta+ \omega, f_{\theta}(x))$,
where $\omega \in (\R \setminus \Q)/\Z$ is called the \textit{frequency}.
\end{defi}
Through out this paper, we  tacitly identify $\T$ with $\R/\Z$.
For a qpf-map $f$, we say that a continuous map $F : \T \times \R \to \T \times \R$ is a \textit{lift} of $f$ if $F(\theta,x) = (\theta+\omega,F_{\theta}(x))$ such that $F_{\theta}(x) \mod \Z = f_{\theta}(x \mod \Z)$ and $F_{\theta}(x+1) = F_{\theta}(x)+1$ for any $(\theta,x) \in \T \times \R$.
Similar to the case of circle homeomorphisms, the limit
$\rho(F) = \lim_{n \to \infty}((F^{n})_{\theta}(x) - x )/n$
exists and is independent of the choice of $(\theta,x)$. Here we  set $(F^n)_{\theta} = F_{\theta+(n-1)\omega}\cdots F_{\theta}$ for all integer $n \geq 1$. The \textit{fibered rotation number} of $f$ is defined as $\rho(f) := \rho(F) \mod \ 1$, and is independent of the choice of $F$.

For any qpf-map $f$ with a lift $F$, for any $t > 0$ we set $F_{t}(\theta,x) = (\theta+\omega, F_{\theta}(x) +t)$.
A qpf-map $f$  is said to be \textit{mode-locked} if $\varepsilon \mapsto \rho(F_{\varepsilon})$ is constant on a neighborhood of $\varepsilon = 0$. This is equivalent to say that $g \mapsto \rho(g)$ is locally constant at $g=f$. In \cite{BJ}, the authors showed that $f$ is mode-locked if and only if there exists a closed annulus, bounded by continuous curves, which is mapped into its own interior by some iterate of $f$. Moreover, they showed that whenever $\omega, \rho(f)$ and $1$ are rationally independent, the map $\varepsilon \mapsto \rho(F_{\varepsilon})$ is strictly monotonically increasing at $\varepsilon = 0$. In particular, in the latter case, $f$ is not mode-locked.

There is a closely related line of research focused on the structure of quasi-periodically forced maps on $\T^2$. In \cite{JS}, the authors showed that any qpf-map $f$ of bounded mean motion such that $\omega, \rho(f),1$ are rationally independent, is semi-conjugate to the irrational torus translation $(\theta,x) \mapsto (\theta+\omega, x + \rho(f))$ via a fibre-respecting semi-conjugacy (see \cite[Theorem 3.1]{JS} and \cite[Theorem 2.3]{BJ} for the formal statement). Such result is crucial for the proof of strictly monotonicity in \cite[Lemma 3.2]{BJ}. J\"ager \cite{J} later generalised  \cite[Theorem 3.1]{JS} to any  minimal totally irrational pseudo-rotations on $\T^2$ with bounded mean motion.  

Aside from the deterministic results mentioned above, it is also natural to study the generic picture of quasi-periodically forced maps. A natural question is:
\
\begin{center} \textit{ Is a generic quasi-periodically forced map on $\T^2$ mode-locked ? } \end{center}
Besides its intrinsic interest, the above question also has roots in the study of differentiable dynamical systems and Schr\"odinger operators. We will elaborate this point in Subsection \ref{More on quasi-periodically / dynamically forced maps}. 

Depending on one's interpretation of the notion of genericity, and also on the regularity of the maps, the answer to the above question may vary.
So far, this question has been studied by different means. In \cite{JW}, J\"ager and Wang used a multi-scale argument originated from the classical works of  Benedieck-Carleson \cite{BC} and Young \cite{Y} to construct a $C^1$ family of quasi-periodically forced \textit{ diffeomorphisms}, among which mode-locked parameter has small measure. In \cite{WZJ}, J\"ager, Wang and Zhou showed that for a topologically generic frequency $\omega$, the set of mode-locked qpf-maps with frequency $\omega$ is residual.  We leave to Subsection \ref{More on quasi-periodically / dynamically forced maps} for more on the genericity condition. For a recent result on the local density of mode-locking in a related setting, we mention \cite{KWYZ}.

In this paper, we prove the topological genericity of mode-locking property for a  much more general class of maps which we call \textit{dynamically forced map}. This class includes circle homeomorphism extensions over any strictly ergodic homeomorphism of a compact manifold. As a special case, we are able to prove the topological genericity of mode-locking for qpf-maps  with any given irrational frequency. Moreover, we show that among the set of qpf-maps which are not mode-locked, a dense subset corresponds to maps which are topologically linearizable. 

Let $X$ be a compact metric space, and let $g: X \to X$ be a uniquely ergodic  homeomorphism with $\mu$ as the unique $g-$invariant measure.
\begin{defi}[$g-$forced maps]\label{def qpf}
Let $g: X \to X$ be given as above.
We say that a homeomorphism $f : X \times \T \to X \times \T$ is a  \textit{ $g-$forced map} if it is of the form
\aryst
f(\theta,y) = (g(\theta), f_{\theta}(y))
\earyst
and admits a \textit{lift}, i.e., a homeomorphism $F: X \times \R \to X \times \R$ satisfying $F_{\theta}(x) \mod \Z = f_{\theta}(x \mod \Z)$  and $F_{\theta}(x+1) = F_{\theta}(x)+1$ for any $(\theta,x) \in X \times \R$. We denote the set of $g-$forced maps by $\cal F_g$. The space $\cal F_g$ is a complete metric space under the  $C^0$ distance $d_{C^0}$. 
\end{defi}

Similar to the case of quasi-periodically forced maps on $\T^2$, for any $g-$forced map $f$ with a lift $F$, for any integer $n \geq 1$, we use the notation
$$(F^n)_{\theta} \quad \myeq  \quad F_{g^{n-1}(\theta)}\cdots F_{\theta}.$$
We define $(f^n)_{\theta}$ in a similar way.
The limit $\rho(F) = \lim_{n \to \infty} ((F^{n})_{\theta}(y) - y)/n$
exists and is independent of the choice of $(\theta, y)$.
It is easy to see that a lift of $f$ is unique up to the addition of a function in $C^0(X,\Z)$ to the $\R-$coordinate. Set $D = \{\int \varphi d\mu \mid \varphi \in C^0(X,\Z)\}$.
Then  the \textit{fibered rotation number} of $f$, defined as $\rho(f) = \rho(F) \mod \ D $, is independent  of the lift $F$.
For any $t \in \R$, we set $f_{t}(\theta,y) = (g(\theta), f_{\theta}(y) + t)$ and $F_{t}(\theta,y) = (g(\theta), F_{\theta}(y) + t)$. It is clear that $F_t$ is a lift of $f_t$ for all $t \in \R$.

Based on the local behaviour of $\rho$, we can give a crude classification for $g-$forced maps as follows.
\begin{defi}\label{def mode lock}
Given a $g-$forced map $f$ with a lift $F$. Then $f$ is said to be
\enmt
\item[$\bullet$] \textit{mode-locked} if the function $\varepsilon \mapsto \rho(F_{\varepsilon})$ is constant on an open neighborhood of $0$;
\item[$\bullet$] \textit{semi-locked} if  it is not mode-locked and the function $\varepsilon \mapsto \rho(F_{\varepsilon})$ is constant on either $(-\varepsilon',0]$ or $[0, \varepsilon')$ for some $\varepsilon' > 0$;
\item[$\bullet$] \textit{unlocked} if the function $\varepsilon \mapsto \rho(F_{\varepsilon})$ is strictly increasing at $0$.
\eenmt
We denote the set of mode-locked (resp. semi-locked, unlocked) $g-$forced maps  by $\cal{ML}_g$ (resp. $\cal{SL}_g, \cal{UL}_g$). Often we omit the subscript  and write $\cal{ML}$, $\cal{SL}$ and $\cal{UL}$ instead. Every $g-$forced map belongs to exactly one of these three subsets.
\end{defi}

Our first result shows that: under a very mild condition on $g$,  the set of $g-$forced maps which are topologically conjugate to $g \times R_{\alpha}$, where $R_{\alpha}$ is the circle rotation $x \mapsto x + \alpha$, is dense in the complement of $\cal{ML}$.

In the following, we let $X$ be a compact metric space, and let $g : X \to X$ be a strictly ergodic (i.e. uniquely ergodic and minimal) homeomorphism with a non-periodic factor of finite dimension. That is, there is a homeomorphism $\bar{g} : Y \to Y$, where $Y$ is an infinite compact subset of some Euclidean space $\R^d$, and there is a onto continuous map $h : X \to Y$ such that $h  g = \bar{g}  h$.  In particular, $g$ can be any strictly ergodic homeomorphsm of a  compact  manifold of positive dimension.

\begin{thm}\label{thm a ternary version}
For $(X, g)$ as above, the following is true.
For any $g-$forced map $f$ that is not mode-locked, for any lift $F$ of $f$,  for any $\varepsilon > 0$, there exists a $g-$forced map $f'$ with $d_{C^0}(f',f) < \varepsilon$, and a homeomorphism $h: X \times \T \to   X \times \T$ of the form $h(\theta,x) = (\theta, h_{\theta}(x))$ with a lift $H: X \times \R \to X \times \R$ (that is, $H(\theta,x) \mod \Z = h_{\theta}(x \mod \Z)$ and $H(\theta, x+1) = H(\theta,x)+1$ for any $(\theta,x) \in X \times \R$), such that
\aryst
f' h(x,y) = h(g(x), y + \rho(F)), \quad \forall (x,y) \in X \times \T.
\earyst
\end{thm}

Theorem \ref{thm a ternary version} can be seen as a nonlinear version of the reducibility result \cite[Theorem 5]{ABD2} for $\mathrm{SL}(2,\R)-$cocycles \footnote{Indeed, it is proved in  \cite[Appendix C]{ABD2} that a $\mathrm{SL}(2,\R)-$cocycle is mode-locked if and only if it is \textit{uniformly hyperbolic}.}. For ${\rm SL}(2,\R)-$cocycles with high regularity over torus translations, KAM method and Renormalization (see \cite{E,AK, AFK}) are effective tools for studying reducibility. While in our case, we use a topological method. 
As an easy but interesting consequence of Theorem \ref{thm a ternary version}, we have the following.
\begin{thm}\label{thm2 a ternary version}
For $(X, g)$ as in Theorem \ref{thm a ternary version}, mode-locked $g-$forced maps form an open and dense subset of the space of $g-$forced maps.
\end{thm}
We prove Theorem \ref{thm2 a ternary version} using Theorem \ref{thm a ternary version} as an intermediate step. This strategy is inspired by \cite{ABD2} in their study of Schr\"odinger operators.
In contrast to the result in \cite{ABD2}, the density of mode-locking for dynamically forced maps is true regardless of the range of the Schwartzman asymptotic cycle $G(g)$ (for its definition, see  \cite[Section 1.1]{ABD2}), while a $\mathrm{SL}(2,\R)-$cocycle $f$ over base map $g$ can be mode-locked only if $\rho(f) \in G(g) \mod \Z$. This is due to the fact that for dynamically forced maps, one can perform perturbations localised in the fiber,  a convenient feature that is not shared by $\mathrm{SL}(2,\R)-$cocycles.

We note that, even in the case where $X = \T$ and $g$ is given by an irrational rotation $g(\theta) = \theta+\omega$, Theorem \ref{thm2 a ternary version} improves the result in \cite{WZJ}. Indeed, in \cite{WZJ} the authors need to require the frequency $\omega$ to satisfy a topologically generic condition. In particular, it was unknown in \cite{WZJ} that  mode-locking could be dense for any Diophantine frequency $\omega$ (see the remark below \cite[Corollary 1.5]{WZJ}). Moreover, our result covers very general base maps.
Thus our Theorem \ref{thm2 a ternary version} can be viewed as a clear strengthening of the main result in \cite{WZJ}. Also, we can deduce from Theorem \ref{thm2 a ternary version} a generalisation of \cite[Theorem 1.6]{WZJ}.

\begin{cor}\label{thm3 a ternary version}
Let $(X,g)$ be as in Theorem \ref{thm a ternary version}, and let $\cal P$ denote the set of continuous maps from $\T$ to $\cal F_g$ endowed with the uniform distance. Then for a topologically generic $\hat{f} \in \cal P$, the function $\tau \mapsto \rho(\hat{f}(\tau))$ is locally constant on an open and dense subset of $\T$.
\end{cor}
\begin{proof}
We fix an abitrary $\tau_0 \in \T$. By  Theorem \ref{thm2 a ternary version},
for any $\hat{f} \in \cal P$,   any $\epsilon > 0$, one can find $f' \in \cal{ML}$ such that $d_{C^0}(f', \hat{f}(\tau_0)) < \epsilon$. Thus by convex interpolation using $f'$ and $\hat{f}(\tau)$ for $\tau$ close to $\tau_0$, and by Corollary \ref{cor ml is open}, we can construct $\hat{f}' \in \cal{P}$ close to $\hat{f}$ so that $\hat{f'}(\tau) \in \cal{ML}$ for every $\tau$ in a neighbourhood of $\tau_0$. It is clear that $\cal P$ is a complete metric space under the uniform distance. We conclude the proof by taking $\tau_0$ over a dense subset of $\T$ \clb and by Baire's category argument.
\end{proof}
It is direct to see that one can adapt the above proof so as to consider maps in $\cal P$ satisfying the twist condition  in \cite[Theorem 1.6]{WZJ}. 

\subsection{Background and further perspective }\label{More on quasi-periodically / dynamically forced maps}
A prominent example of qpf-map is the Arnold circle map,
\aryst
f_{\alpha, \beta, \tau}: \T^2 \to \T^2, \quad (\theta,x) \mapsto (\theta+\omega, x + \tau + \frac{\alpha}{2\pi}\sin(2\pi x) + \beta g(\theta) \mod \ 1),
\earyst
with parameters $\alpha \in [0,1], \tau, \beta \in \R$ and a continuous forcing function $g : \T \to \R$. It was introduced in \cite{DGO} as a simple model of an oscillator forced at two incommensurate frequencies. Mode-locking was observed numerically on open regions in the $(\alpha,\tau)-$parameter space, known as the Arnold tongues.\clb

Another well-known class of qpf-maps is the so-called (generalised) quasi-periodic Harper map,
\aryst
s_{E} : \T \times \overline{\R} \to \T \times \overline{\R}, \quad (\theta,x) \mapsto (\theta + \omega, V(\theta) - E - \frac{1}{x})
\earyst
where $V : \T \to \R$ is a continuous function and $E \in \R$. Here we use the identification $\overline{\R} \simeq \T$ to simplify the notations. This class of map arises  naturally in the study of $1D$ discrete  Schr\"odinger operators $(H_{\theta}u)_n = u_{n+1} + u_{n-1} + V(\theta + n \omega) u_n$. It is by-now well-understood that the spectrum of $H_{\theta}$ equals to the set of parameter $E$ such that, as a projective cocycle, $s_{E}$ is not uniformly hyperbolic. More refined links relating the dynamics of $s_{E}$ with the spectral property of $H_{\theta}$ have been found, for instance in \cite{A}. We refer the readers to \cite{D, Jit} for related topics as there is a vast literature dedicated to Schr\"odinger operators.

Just as qpf-maps are special cases of our dynamically forced maps, many results on quasi-periodic Harper maps was generalised to the general $\mathrm{SL}(2,\R)-$cocycles, motivated by the study of Schr\"odinger operators with general potential functions and the techniques therein. In \cite{ABD2}, Avila, Bochi and Damanik showed that the spectrum of a $1D$ Schr\"odinger operator with a $C^0$ generic potential generated by a map $g$ as in Theorem \ref{thm a ternary version} has open gaps at all the labels, which extends their previous result \cite{ABD} on Cantor spectrums. One may ask whether mode-locking  could be generic in much higher regularities. To this end,  Chulaevsky and Sinai  in \cite{CS} suggested that in contrast to the circle rotations, for translations on the two-dimensional torus the spectrum can be an interval for generic large smooth potentials (see also \cite[Introduction]{GSV}), in which case the mode-locked parameters for the Schr\"odinger cocycles were not dense. In the analytic category,  Goldstein, Schlag and Voda \cite[Theorem A, Remark 1.2(b)]{GSV} recently proved that for a multi-dimensional shift on torus with Diophantine frequency, the spectrum is an interval for almost every large trigonometric polynomial potential. This 
 hints  the failure of the genericity of mode-locking in higher regularity and higher dimension (see also \cite{JW}). Nevertheless, by combining  the method of the present paper and  ideas from Bochi's work \cite{B} on Lyapunov exponents, we can show that mode-locking for multi-frequency forced maps remains generic under some  stronger topology, such as the one induced by the norm $d_{C^{0,\infty}}(f,g) = \sup_{\theta\in X} d_{C^{\infty}}(f_{\theta},g_{\theta})$. We will treat this in a separate note.
\subsection{Idea of the proof.}
Our idea is originated from \cite{ABD, ABD2}.
Given $f$, a $g-$forced map which is not mode-locked,
we will perturbe $f$ into a direct product, modulo conjugation. This is rest on the basic observation that for any unlocked map, a small perturbation can promote linear displacement  for the iterates of any given point.

A key observation in the proof of Theorem \ref{thm a ternary version} is that the mean motion, which is sublinear in time, can be cancelled out by a small perturbation. The global perturbation is divided into local perturbations in Section \ref{A perturbation lemma} at finitely many stages, and is organised using the dynamical stratification in Section \ref{sect Stratification}. Theorem \ref{thm2 a ternary version} is an easy consequence of Theorem \ref{thm a ternary version}. Theorem \ref{thm a ternary version} and  \ref{thm2 a ternary version} are proved in Section \ref{Proof of the main results}.

\smallskip
\noindent{\textbf{Notation.}} Given a subset $M \subset X$, we denote by $int(M)$ the interior of $M$, and denote by $\overline{M}$ the closure of $M$ in $X$. We denote by $B(M,r)$ the $r-$open neighbourhood of $M$ in $X$ for any $r > 0$.
We say that a function $\varphi \in C^0(\R, \R)$ is strictly increasing at $t=t_0$ if $\varphi(t_0-\epsilon) < \varphi(t_0) < \varphi(t_0+ \epsilon)$ for any $\epsilon > 0$.

\smallskip

\noindent{\textbf{Acknowledgement.}}
The author  thanks Artur Avila for asking him the question and useful conversations, and thanks Tobias J\"ager for useful conversations and telling him a mistake on the definition of fibered rotation number in an earlier version of this paper. The author also thanks Kristian Bjerkl\"ov for a related reference.

\section{Preliminary} \label{prelim}

\subsection{Basic properties of dynamically forced maps}
In this subsection, we only require $g$ to be a uniquely ergodic homeomorphism of a compact topological space $X$.
We start by giving the following basic relations between $\cal{ML}, \cal{SL}$ and $\cal{UL}$.
\begin{lemma}\label{lem basic on ml sl ul} 
For any $g-$forced map $f \in \cal{SL}$, any $\varepsilon > 0$, there exists $f' \in \cal{ML}$ and  $f'' \in  \cal{UL}$ such that $d_{C^{0}}(f,f'),  d_{C^{0}}(f,f'') < \varepsilon$.
\end{lemma}
\begin{proof}
Let $F$ be a lift of $f$. Without loss of generality, we assume that there exists $\varepsilon_0 \in (0, \varepsilon)$ such that $\rho(F_{\varepsilon'})  = \rho(F)$ for all $ \varepsilon' \in (-\varepsilon_0,0)$; and $\rho(F_{\varepsilon'}) > \rho(F)$ for all $\varepsilon' \in (0,\varepsilon_0)$. 

By choosing $f' = f_{\varepsilon'}$ for some $\varepsilon' \in (-\varepsilon_0, 0)$ sufficiently close to $0$, we can ensure that $d_{C^0}(f',f)<\varepsilon$. It is clear that $f' \in \cal{ML}$ by Definition \ref{def mode lock}. On the other hand, for each $\varepsilon' \in (0, \varepsilon_0)$ such that $f_{\varepsilon'} \in \cal{ML} \bigcup \cal{SL} = \cal{UL}^c$, there exists an open interval $J \subset (0, \varepsilon_0)$ so that $\rho(F_{\varepsilon'}) = \rho(F_{\varepsilon''}), \forall \varepsilon'' \in J$. This immediately implies that the set $A := \{ \rho(F_{\varepsilon'}) \mid \mbox{$\varepsilon' \in (0, \varepsilon_0)$ and $f_{\varepsilon'} \notin \cal{UL}$}\}$ is countable. Moreover, by our hypothesis on $F$, $\rho(F) \notin A$. Thus there exists arbitrarily small $\varepsilon'' \in (0, \varepsilon_0)$ with $f_{\varepsilon''} \in \cal{UL}$. This concludes the proof.
\end{proof}

\begin{lemma}\label{lem ml is open} 
A $g-$forced map $f$ with a lift $F$ is in $\cal{ML}$ if and only if there exists $\varepsilon > 0$ such that for any $F'$ which is a lift of a $g-$forced map and satisfies $d_{C^0}(F,F') < \varepsilon$, we have $\rho(F) =\rho(F')$. 
\end{lemma}
\begin{proof} We just need to show the \lq\lq only if'' part, for the other direction is obvious.
By Definition \ref{def mode lock}, there exists $\epsilon > 0$ such that $\rho(F_{-\epsilon} )= \rho(F) = \rho(F_{\epsilon})$. Take $\varepsilon \in (0,\epsilon)$. Then $(F_{-\epsilon})_{\theta}(y) < F'_{\theta}(y) <( F_{\epsilon})_{\theta}(y)$ for  any $F'$ in the lemma, and any $(\theta,y) \in X \times \R$. By monotonicity, we have $\rho(F_{-\epsilon}) \leq \rho(F') \leq \rho(F_{\epsilon})$. This ends the proof.
\end{proof}
\begin{cor}\label{cor ml is open}
 The set $\cal{ML}$ is open. Moreover, for any homeomorphism $h: X \times \T \to   X \times \T$ of the form $h(\theta,x) = (\theta, h_{\theta}(x))$ with a lift $H: X \times \R \to X \times \R$ (see Theorem \ref{thm a ternary version}),  we have $f \in \cal{ML}$ if and only if $hfh^{-1} \in \cal{ML}$.
\end{cor}
\begin{proof}
We claim that for any $f \in \cal{ML}$ with a lift $F$, for any $f' \in \cal{F}_g$ with $d_{C^0}(f,f') < 1/4$, there exists a lift of $F'$ of $f'$ such that $d_{C^0}(F,F') = d_{C^0}(f,f')$. Indeed, we take an arbitrary lift $F''$ of $f'$. For any $\theta \in X$, $F''_{\theta}$ is a lift of $f'_{\theta}$. Thus there is a function $\phi : X \to \Z$ such that $d_{C^0}(F''_{\theta}+\phi(\theta), F_{\theta}) = d_{C^0}(f'_{\theta}, f_{\theta})$ for any $\theta \in X$. By  continuity and $d_{C^0}(f,f') < 1/4$, $\phi$ is continuous. It suffices to take  $F'(\theta,x)= (g(\theta), F''_{\theta}(x)+\phi(\theta))$.

By Lemma \ref{lem ml is open}, the above claim immediately implies that $\cal{ML}$ is $C^0$ open. For the second statement, we note that for any lift $F$ of $f$, $HFH^{-1}$ is a lift of $hfh^{-1}$. Then we obtain $hfh^{-1} \in \cal{ML}$ by Lemma \ref{lem ml is open} and Definition \ref{def mode lock}.
\end{proof}

\begin{defi}\label{def fibered rotation number}
\label{def fibered rotation number}
For any $g-$forced map $f$ with a lift  $F$, for any integer $n > 0$ we set
\aryst
\underline{M}(F,n) = \inf_{(\theta, y) \in X \times \R} ((F^n)_{\theta}(y) - y) \ \ \mbox{and} \ \ 
\overline{M}(F,n) = \sup_{(\theta, y) \in X \times \R} ((F^{n})_{\theta}(y) - y ).
\earyst
\end{defi}
We collect some general properties of $g-$forced maps. The following is an immediate consequence of the unique ergodicity of $g$. We  omit the proof.
\begin{lemma}\label{lem mean motin o(N)}
Given a $g-$forced map $f$ with a lift  $F$. For any $\kappa_0 > 0$, there exists   $N_0 = N_0(F, \kappa_0) > 0$  such that for any $n > N_0$, we have $[\underline{M}(F,n),\overline{M}(F,n)] \subset n\rho(F) + (-n \kappa_0, n \kappa_0)$.
\end{lemma}

We also have the following.
\begin{lemma}\label{lem promoting displacement}
Given a $g-$forced map $f$ with a lift  $F$.
If for some constant $\epsilon > 0$, we have $\rho(F_{-\epsilon}) < \rho(F) < \rho(F_{\epsilon})$, then there exist $\kappa_1 = \kappa_1(F, \epsilon) > 0$ and an integer $N_1 = N_1(F, \epsilon) > 0$ such that for any $n > N_1$, we have
\aryst
\overline{M}(F_{-\epsilon}, n) < \underline{M}(F, n) - n \kappa_1 < \overline{M}(F,n) + n\kappa_1 < \underline{M}(F_{\epsilon}, n).
\earyst
\end{lemma}
\begin{proof}
The first and the third inequality are immediate consequences of Lemma \ref{lem mean motin o(N)} and our hypotheses. The second inequality is obvious.
\end{proof}

\subsection{An inverse function theorem}

The following lemma is simple but convenient for constructing perturbations that depend continuously on parameters.
\begin{lemma} \label{lem inverse function}
Given $\varepsilon > 0$ and  a  metric space  $Y$. Let $\Sigma \subset Y \times \R$ be an open set so that $\pi_{Y}(\Sigma) = Y$, and let $\varphi : (-\varepsilon,\varepsilon) \times Y \to \R$ be a continuous function such that
\enmt
\item for any $x \in Y$, the function $t \mapsto \varphi(t,x)$ is monotonically increasing, and 
\aryst
\{z \in \R \mid (x,z) \in \Sigma\} \subset \varphi((-\varepsilon,\varepsilon) \times \{ x \}),
\earyst
\item for any $(t_0, x) \in (-\varepsilon,\varepsilon) \times Y$ such that $(x, \varphi(t_0,x)) \in \Sigma$, the function $t \mapsto \varphi(t,x)$ is strictly increasing at $t=t_0$.
\eenmt
Then there exists a continuous map $s : \Sigma \to (-\varepsilon,\varepsilon)$ such that
\aryst
\varphi(s(x,z), x) = z, \quad \forall (x,z) \in \Sigma.
\earyst
\end{lemma}
\begin{proof}
By (1), for any $(x,z) \in \Sigma$, there exists $t \in (-\varepsilon, \varepsilon)$ such that $\varphi(t, x) = z$. By (1),(2), we see that such $t$ is unique. Thus there exists a unique function $s : \Sigma \to (-\varepsilon, \varepsilon)$ such that $\varphi(s(x,z), x) = z$ for any $(x,z) \in \Sigma$. It remains to show that $s$ is continuous.

Given an arbitrary $(x,z) \in \Sigma$, let $t_0 = s(x,z)$. Then by $\varphi(t_0, x) = z$ and (2), we see that the function $t \mapsto \varphi(t,x)$ is strictly increasing at $t=t_0$. Thus for any $\epsilon \in (0, \varepsilon - t_0)$, there exists $\delta > 0$ such that $\varphi(t_0 + \epsilon, x) > z + 2\delta$. Since $\varphi$ is continuous, there exists $\kappa > 0$ such that for any $x' \in B(x, \kappa)$, $\varphi(t_0 + \epsilon, x') > z + \delta$. Then for any $(x',z') \in \Sigma \cap (B(x,\kappa) \times (z-\delta, z+\delta))$, we have $s(x',z') < t_0 + \epsilon$. Indeed, if $s(x',z') \geq t_0 + \epsilon$, then we would have a contradiction since
\aryst
z' = \varphi(s(x',z'), x') \geq \varphi(t_0 + \epsilon, x') > z + \delta > z'.
\earyst
 By a similar argument, we can see that for any $\epsilon \in (0, t_0 + \varepsilon )$, there exist $\kappa',\delta' > 0$ such that $s(x',z') > t_0 - \epsilon$ for any $(x',z') \in \Sigma \cap ( B(x, \kappa') \times (z - \delta', z+\delta'))$. This shows that $s$ is continuous and thus concludes the proof.
\end{proof}

\section{Stratification}\label{sect Stratification}

Let $g : X \to X$ be given as in Theorem \ref{thm a ternary version}. That is, $X$ is a compact metric space, and $g$ is a strictly ergodic homeomorphism with a non-periodic factor of finite dimension.
As in \cite{ABD2}, for any integers $n,N,d > 0$,  a compact subset $K \subset X$ is said
\enmt
\item \textit{$n-$good} if $g^{k}(K)$ for $0 \leq k \leq n-1$ are disjoint subsets,
\item \textit{$N-$spanning} if the union of $g^{k}(K)$ for $0 \leq k \leq N-1$ covers $X$,
\item \textit{$d-$mild} if for any $x \in X$, $\{g^{k}(x) \mid k \in \Z\}$ enters $\partial K$ at most $d$ times.  
\eenmt

The following is contained in \cite[Proposition 5.1]{ABD2}.
\begin{lemma} \label{lem choose returning set}
There exists an integer $d > 0$ such that for every integer $n > 0$, there exist an integer $D > 0$ and a compact subset $K \subset X$  that is  $n-$good, $D-$spanning and $d-$mild.
\end{lemma}

Let  $K \subset X$ be a  $n-$good, $M-$spanning and $d-$mild compact subset. For each $x \in X$, we set
\aryst
&&\ell^{+}(x) = \min\{j > 0 \mid g^{j}(x) \in int(K)\}, \quad \ell^{-}(x) = \min\{j \geq 0 \mid g^{-j}(x) \in int(K) \}, \\
&&\ell(x) = \min\{j > 0 \mid g^{j}(x) \in K \}, \\
&&T(x) = \{j \in \Z \mid -\ell^{-}(x)  < j < \ell^{+}(x)\}, \quad T_B(x) = \{j \in T(x) \mid g^{j}(x) \in \partial K \}, \\
&&N(x) = \#T_B(x), \quad K^{i} = \{x \in K \mid N(x) \geq d-i\},\forall -1 \leq i \leq d.
\earyst
Let $Z^{i} = K^{i} \setminus K^{i-1} = \{x \in K \mid N(x) = d-i\}$ for each $0 \leq i \leq d$.

Lemma \ref{lem choose returning set} and the notations introduced above are minor modifications of those in the proof of \cite[Lemma 4.1]{ABD2}. We will use them in Section \ref{Proof of the main results} to construct a global perturbation from its local counterparts. In \cite{ABD2}, it is important to give a good upper bound for the constant $D$ appeared in Lemma \ref{lem choose returning set}, while in our case, we do not need such bound.
In the following, 
we collect some facts from \cite{ABD2}. 
\begin{lemma}\label{lemma fact} We have
\enmt 
\item For any $x \in K$, $\ell(x) \leq \ell^{+}(x)$ and $n \leq  \ell(x) \leq D$,
\item $T$ and  $T_B$ are upper-semicontinuous,
\item $T$ and $T_B$, and hence also $\ell$, are locally constant on $Z^{i}$,
\item $K^i$ is closed for all $-1 \leq i \leq d$ and $\emptyset = K^{-1}  \subset K^{0} \subset \cdots \subset K^{d} = K$,
\item For any $x \in K^{i}$, any $0 \leq m < \ell^{+}(x)$ such that $g^{m}(x) \in K$, we have $g^{m}(x) \in K^{i}$.
\eenmt
\end{lemma}
\begin{proof}
Items (1),(2) are immediate consequences of the definition. Item (3) is essentially contained in the proof of the claim in  \cite[Proof of Lemma 4.1]{ABD2}. Item (4) follows from (2) and the fact that $K$ is $d-$mild. Item (5) is true since for any such $m$, we have $\{g^{j}(x) \mid j \in T(x)\} = \{g^{j}(g^{m}(x)) \mid j \in T(g^{m}(x))\}$, and as a consequence, $N(x) =N(g^m(x))$.
\end{proof}

In the rest of this paper, we denote by $\widetilde{Homeo}$ the set of  $1-$periodic orientation preserving homeomorphisms of $\R$. In other words, $\widetilde{Homeo}$ is the set of maps  which are  lifts of maps in the group $Homeo_+(\T)$ of orientation preserving homeomorphisms of $\T$. Given any integer $k \geq 1$ and $g_1,\cdots, g_k \in \widetilde{Homeo}$, we denote by $\prod_{i=1}^{k} g_i $ the map $ g_k \cdots g_1$. We set $\prod_{i=1}^{0} g_i =\Id$.

\section{Perturbation lemmata}\label{A perturbation lemma}
In this section, we fix a constant $\epsilon  \in ( 0, \frac{1}{4})$ and  a $g-$forced map $f$ with a lift $F$, satisfying
\ary \label{strictmono}
\rho(F_{-\epsilon}) < \rho(F) < \rho(F_{\epsilon}).
\eary 
We stress that $f$ is not necessarily unlocked. 
\subsection{Cancellation of the mean motion}
For any $\kappa > 0$ and  any integer $N > 0$, let
\aryst
\Gamma_N(F, \kappa) := \{ (\theta,y,z) \in X \times \R \times \R  \mid  |z-(F^{N})_{\theta}(y)| < N \kappa \}.
\earyst
\begin{lemma} \label{lem cancel mm}
Let $\kappa_1 = \kappa_1(F,\epsilon)> 0$ and $N_1 = N_1(F, \epsilon) > 0$ be given by Lemma \ref{lem promoting displacement}. Then for any $g-$forced map $\check{f}$ with a lift $\check{F}$ such that $d_{C^0}(\check{F},F) < \epsilon$, for any integer $N \geq N_1$, there exists a continuous map $\Phi^{\check{F}}_N : \Gamma_N(F, \kappa_1) \to \widetilde{Homeo}^{N}$ such that the following is true. For any $(\theta, y, z) \in \Gamma_N(F, \kappa_1)$,  let $\Phi^{\check{F}}_N(\theta,y,z) = (G_0,\cdots, G_{N-1})$, then
\enmt
\item $d_{C^0}(G_i, \check{F}_{g^{i}(\theta)}) < 2\epsilon$ for any $0 \leq i \leq N-1$, 
\item $G_{N-1}\cdots G_0(y) = z$,
\item if $z = (\check{F}^{N})_{\theta}(y)$, then $G_i = \check{F}_{g^{i}(\theta)}$ for any $0 \leq i \leq N-1$.
\eenmt
\end{lemma}
\begin{proof}
It is clear that for any $\theta \in X, y \in \R$, the function $\epsilon'  \mapsto  (\check{F}_{\epsilon'}^{N})_{\theta}(y)$ is strictly increasing.
By $d_{C^0}(\check{F}, F)  < \epsilon$,  we have $(\check{F}_{2\epsilon})_{\theta}(y) \geq (F_{\epsilon})_{\theta}(y)$ for any $\theta \in X$ and $y \in \R$. Then  by Lemma \ref{lem promoting displacement},
\aryst
 (\check{F}_{2\epsilon}^{N})_{\theta}(y)  \geq (F_{\epsilon}^{N})_{\theta}(y) >   (F^{N})_{\theta}(y) + N\kappa_1.
\earyst
Similarly, we have $(\check{F}^{N}_{-2\epsilon})_{\theta}(y)  < (F^{N})_{\theta}(y) - N\kappa_1$.
Then we can verify (1),(2) in Lemma \ref{lem inverse function} for $( 2\epsilon, X \times \R, \Gamma_N(F,\kappa_1))$ in place of $( \varepsilon, Y, \Sigma)$, and for function $(\epsilon', \theta, y) \mapsto (\check{F}_{\epsilon'})_{\theta}^{N}(y)$ in place of $\varphi$.
By  Lemma \ref{lem inverse function} there exists a continuous function $s : \Gamma_N(F, \kappa_1) \to  (-2\epsilon, 2\epsilon)$ such that $\Phi^{\check{F}}_N$ defined by
\aryst
\Phi^{\check{F}}_N(\theta,y,z):=((\check{F}_{s(\theta, y, z)})_{g^{i}(\theta)})_{i=0}^{N-1}
\earyst
satisfies our lemma.
\end{proof}

\subsection{Local perturbation}
For any $g-$forced map $f'$ with a lift denoted by $F'$, for any integer $N > 0$, we define
\aryst
\Omega_N(F') &:=& \{(\theta, \underline{y}, y, \overline{y}, z) \in X \times \R^4 \mid \underline{y} < y < \overline{y} \leq \underline{y}+1, z \in ((F'^{N})_{\theta}(\underline{y}), (F'^{N})_{\theta}(\overline{y}))  \}, \\
\overline{\Omega}_{N} &:=&\{(\theta, \underline{y}, y, \overline{y}) \in X \times \R^3 \mid \underline{y} < y < \overline{y} \leq \underline{y}+1 \}.
\earyst
\begin{lemma} \label{lem local pert} 
There exist $\kappa_2 = \kappa_2(F, \epsilon) \in (0, \epsilon)$ and $N_2 = N_2(F,\epsilon) > 0$ such that the following is true. For any $g-$forced map $\check{f}$ with a lift $\check{F}$ such that $d_{C^0}(\check{F},F) < \kappa_2$,  for any integer $N \geq N_2$, there exists a continuous map $\Psi^{\check{F}}_{0, N} : \Omega_N(\check{F}) \to \widetilde{Homeo}^{N}$ with the following properties: for any $(\theta, \underline{y}, y, \overline{y}, z) \in \Omega_N(\check{F})$, let $\Psi^{\check{F}}_{0, N}(\theta, \underline{y}, y, \overline{y}, z) = (G_0, \cdots, G_{N-1})$, then
\enmt
\item $d_{C^0}( G_i, \check{F}_{g^i(\theta)} ) < 2\epsilon$ for any $0 \leq i \leq N-1$,
\item $G_{N-1} \cdots G_0(y) = z$,
\item for any $1 \leq i \leq N$, for any $w \in [\overline{y}, \underline{y}+1]$, we have  $G_{i-1} \cdots G_{0}(w) = (\check{F}^{i})_{\theta}(w)$,
\item if $z = (\check{F}^{N})_{\theta}(y)$, then $G_i = \check{F}_{g^{i}(\theta)}$ for any $0 \leq i \leq N-1$.
\eenmt 
\end{lemma}
\begin{proof}
Let $\kappa_1 = \kappa_1(F,\epsilon), N_1=N_1(F,\epsilon)$ be given by Lemma \ref{lem promoting displacement}, and let $N_2 =  \max(10\kappa_1^{-1}, N_1)$.
We let $\kappa_2 = \kappa_2(F,\epsilon) \in (0,\epsilon)$ be  sufficiently small such that for any $n > N_2$, for any $f'$ with a lift $F'$ such that $d_{C^0}(F',F) < \kappa_2$, we have
\ary \label{term 201}
A_n :=\sup_{\theta \in X, y \in \R}| (F'^{n})_{\theta}(y) - (F^n)_{\theta}(y) | < \frac{1}{2}\kappa_1 n.
\eary
The existence of $\kappa_2$ is guaranteed by the compactness of $X$, and the formula 
$\lfloor A_{n+m} \rfloor \leq \lfloor A_{m} \rfloor + \lfloor A_{n} \rfloor +2$ for any integers $n,m \geq 1$.

For each $1 \leq i \leq N$, we set
\aryst
\underline{y}_i := (\check{F}^{i})_{\theta}(\underline{y}) \quad \mbox{and} \quad  \overline{y}_i := (\check{F}^{i})_{\theta}(\overline{y}).
\earyst
It is clear that $\underline{y}_i < \overline{y}_i$ for every $1 \leq i \leq N$, and $\underline{y}_N < z < \overline{y}_N$ for any $(\theta, \underline{y}, y, \overline{y}, z) \in \Omega_N(\check{F})$.

Let $\Delta_2 := \{(y,z) \in \R^2 \mid y < z \leq y+1\}$. We define a continuous map $\phi_{\epsilon} : (-1, 1) \times \Delta_2 \to \widetilde{Homeo}$ as follows. For every $t \in (-1,1)$, every $(y,z) \in \Delta_2$, we define $\phi_{\epsilon}(t, y, z)(x)$ for $x \in (y,z)$ by the following formula,
\aryst
\phi_{\epsilon}(t, y, z)(x) = \begin{cases} 
x + t \min(z-x, 2\epsilon), & \forall x \in [y+(1-t)(z-y), z), t \in (0,1), \\
y + (x-y)(1+\frac{t}{1-t}\min(t, \frac{2\epsilon}{z-y})), & \forall x \in (y, y+(1-t)(z-y) ), t \in (0,1), \\
x, & \forall x \in (y,z), t = 0, \\
x + t \min(x-y, 2\epsilon), & \forall x \in (y, z - (t+1)(z-y)] ,  t \in (-1,0),  \\
z - ( z - x )( 1 - \frac{t}{t+1}\min(-t, \frac{2\epsilon}{z-y})) , &\forall x \in (z - (t+1)(z-y), z), t \in (-1,0).
 \end{cases}
\earyst
The following is an illustration for $\phi_{\epsilon}(t,y,z)$ with $t \in [0,1)$.  

\begin{figure}[ht!]
\centering
\includegraphics[width=70mm]{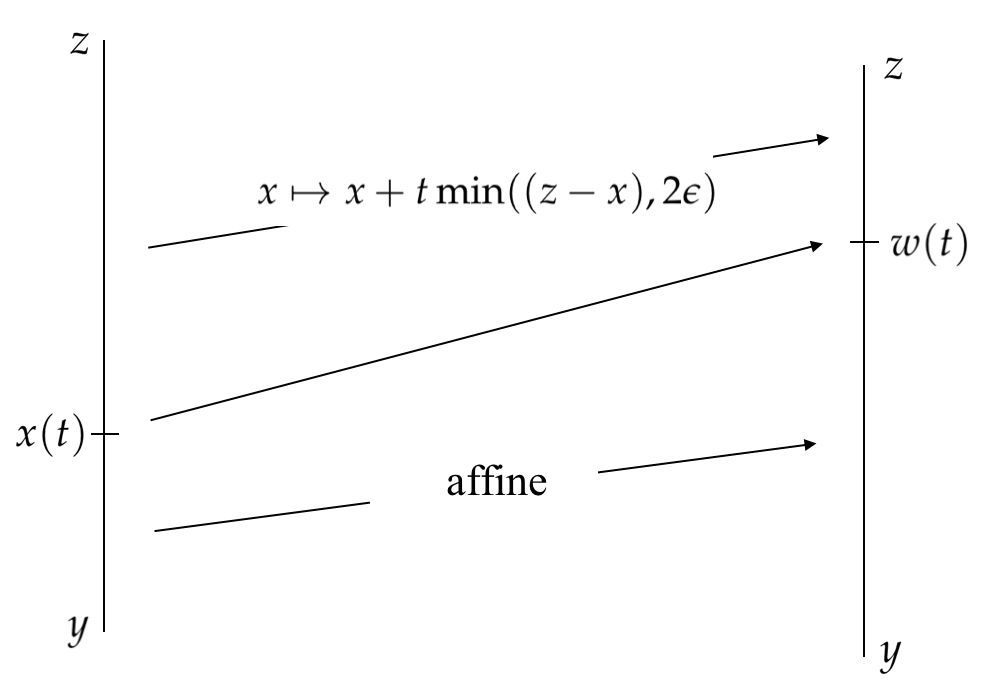}
\caption{
In the above picture, we set $x(t) = y + (1-t)(z-y)$ and $w(t) = y + (1-t)(z-y) + t \min(t(z-y),2\epsilon)$. 
}
\end{figure}

By straightforward computations, we see that $\phi_{\epsilon}(t,y,z)(x)$ is continuous in $t \in (-1,1)$, $(y,z) \in \Delta_2$ and $x \in (y,z)$. Moreover, for any $t \in (-1,1)$, $(y,z) \in \Delta_2$, the restriction of the map $x \mapsto \phi_{\epsilon}(t,y,z)(x)$ to $x \in (y,z)$ is a homeomorphism of $(y,z)$.
We then define $\phi_{\epsilon}(t,y,z)(x)$ for $x \in \Z + (y,z)$ by periodicity. 
For any $t \in (-1,1)$, we have
\aryst
\lim_{x \in (y,z), x \to y} \phi_{\epsilon}(t,y,z)(x) = y \quad \mbox{and} \quad \lim_{x \in (y,z), x \to z} \phi_{\epsilon}(t,y,z)(x) = z.
\earyst
Thus by letting $\phi_{\epsilon}(t,y,z)(x) = x$ for $x \in  \R \setminus (\Z + (y,z))$, we obtain a continuous map $\phi_{\epsilon}$ valued in $\widetilde{Homeo}$.

We can verify that the map $\phi_{\epsilon}$ satisfies the following properties:
\enmt
\item[$(a1)$] for any $t \in (-1,1), (y,z) \in \Delta_2$, $d_{C^0}(\phi_{\epsilon}(t,y,z), \Id) < 2\epsilon$,
\item[$(a2)$]  the function $t \mapsto \phi(t, y, z)(x)$ is \textit{strictly}  increasing for any $x \in (y, z)$, 
\item[$(a3)$] for any $(y,z) \in \Delta_2$, for any $x \in (y,z)$, we have $\lim_{t \to 1}\phi_{\epsilon}(t,y,z)(x) \in \{ x + 2\epsilon, z\}$ and $\lim_{t \to -1}\phi_{\epsilon}(t,y,z)(x) = \{ x - 2\epsilon, y \}$,
\item[$(a4)$] for any $t \in (-1,1)$, any $x \in \Z + [z, y +1]$, we have $\phi_{\epsilon}(t, y, z)(x) = x$,
\item[$(a5)$] for any $(y,z) \in \Delta_2$, we have $\phi_{\epsilon}(0, y, z) = \Id$.
\eenmt 
Indeed, to verify the above statements for $t \in (0,1)$, let $x(t)$ and $w(t)$ be as in Figure 1, then $(a1)$ follows from $x(t) \leq w(t) < x(t) + 2\epsilon$; $(a2)$ follows from that both $x+t\min((z-x), 2\epsilon)$ and $(w(t)-y)/(x(t)-y)$ are strictly increasing in $t$; $(a3)$ follows from noting that $x(t) \to y$ as $t \to 1$; $(a4)$ and $(a5)$ are direct to see. For $t \in (-1,0]$, we can obtain these properties in a similar way.

Let $N \geq N_2$ be in the lemma. Define a continuous map $\varphi : (-1,1) \times \overline{\Omega}_{N} \to \R$ by
\aryst
\varphi(t, \theta, \underline{y}, y, \overline{y}) = \prod_{i=1}^{N}[ \phi_{\epsilon}(t, \underline{y}_i, \overline{y}_i)\check{F}_{g^{i-1}(\theta)}](y).
\earyst

We have the following,

\textbf{Claim:}  we have $\lim_{t \to 1}\varphi(t, \theta, \underline{y}, y, \overline{y}) = (\check{F}^N)_{\theta}(\overline{y})$.
\begin{proof}
Let $y'_0 := y$, and for each $1 \leq k \leq N$, let
\aryst
 y'_k := \lim_{t \to 1} \prod_{i=1}^{k}[\phi_{\epsilon}(t, \underline{y}_i, \overline{y}_i)\check{F}_{g^{i-1}(\theta)}](y).
 \earyst
It is clear that $y'_k > \underline{y}_k$ for every $1 \leq k \leq N$.
Assume that for every $1 \leq k \leq N$, we have $y'_k < \overline{y}_k$.
Then by property $(a3)$ above, for every $1 \leq k \leq N$ we would have
\aryst
y'_{k} = \lim_{t \to 1}\phi_{\epsilon}(t, \underline{y}_k, \overline{y}_k)\check{F}_{g^{k-1}(\theta)}(y'_{k-1}) = \check{F}_{ g^{k-1}(\theta)}(y'_{k-1}) + 2\epsilon =(\check{F}_{2\epsilon})_{ g^{k-1}(\theta)}(y'_{k-1}).
\earyst
Then by \eqref{term 201},  $d_{C^0}(\check{F},F) < \kappa_2 < \epsilon$ and $N>N_2=\max(10\kappa_1^{-1}, N_1)$, we obtain
\aryst
&&\lim_{t \to 1}\prod_{i=1}^{N}[ \phi_{\epsilon}(t, \underline{y}_i, \overline{y}_i)\check{F}_{g^{i-1}(\theta)}](y) = (\check{F}_{2\epsilon}^{N})_{\theta}(y) \geq (F_{\epsilon}^N)_{\theta}(y) \\
& > &  (F^{N})_{\theta}(y) + N\kappa_1 > (F^{N})_{\theta}(\overline{y})+\frac{1}{2}N\kappa_1 > (\check{F}^{N})_{\theta}(\overline{y}).
\earyst
But this contradicts $y < \overline{y}$ and $(a4)$.
Thus there exists $1 \leq k \leq N$ such that $ y'_k = \overline{y}_k$.
Then by $(a3)$,$(a4)$ and $(a5)$, we have
\aryst
(\check{F}^{N})_{\theta}(\overline{y}) \geq\lim_{t \to 1}\varphi(t, \theta, \underline{y}, y, \overline{y}) \geq \prod_{i=k+1}^{N}\check{F}_{g^{i-1}(\theta)}(y'_k) = \overline{y}_N = (\check{F}^{N})_{\theta}(\overline{y}).
\earyst
This concludes the proof of the claim.
\end{proof}
Similar to the above claim, we also have $\lim_{t \to -1}\varphi(t, \theta, \underline{y}, y, \overline{y}) = \check{F}_{\theta}^{N}(\underline{y})$. We thereby verify condition (1) in Lemma \ref{lem inverse function} for map $\varphi$, and for $(1, \overline{\Omega}_N, \Omega_N(\check{F}))$ in place of $(\varepsilon,Y,\Sigma)$. Moreover, by $(a2)$ and $(a4)$, we can  directly verify that the function $t \mapsto \varphi(t, \theta, \underline{y}, y, \overline{y})$ is strictly increasing for any $(\theta, \underline{y}, y, \overline{y}) \in \overline{\Omega}_N$. This verify condition (2) in Lemma \ref{lem inverse function}.
Then by Lemma \ref{lem inverse function},  we obtain a continuous function $s : \Omega_N(\check{F}) \to (-1,1)$ such that for any $(\theta, \underline{y}, y, \overline{y}, z) \in \Omega_N(\check{F})$, we have
\aryst
\varphi(s(\theta, \underline{y}, y, \overline{y}, z), \theta, \underline{y}, y, \overline{y}) = z.
\earyst
We define $\Psi^{\check{F}}_{0, N}(\theta, \underline{y}, y, \overline{y}, z)=(G_0,\cdots, G_{N-1})$ where for each $0 \leq i \leq N-1$
\aryst
G_i = \phi_{\epsilon}(s(\theta, \underline{y}, y, \overline{y}, z), \underline{y}_{i+1}, \overline{y}_{i+1})\check{F}_{g^{i}(\theta)}.
\earyst
We can see that (1)-(4) follows from $(a1)$-$(a5)$ and Lemma \ref{lem inverse function}.
\end{proof}

\subsection{Concatenation}
In this section, we briefly denote $\rho(F)$ by $\rho$.
\begin{prop}\label{prop concatenation}
 There exist $\kappa_3 = \kappa_3(F,\epsilon) \in ( 0, \epsilon)$, $N_3 = N_3(F, \epsilon ) > 0$ such that  for any $g-$forced map $\check{f}$ with a lift $\check{F}$ satisfying $d_{C^0}(\check{F}, F) < \kappa_3$, for any integer $N \geq N_3$, there exists a continuous function $\Psi^{\check{F}}_N : X \to \widetilde{Homeo}^{N}$ such that the following is true: for any $\theta \in X$, let $\Psi^{\check{F}}_N(\theta) = (G_0, \cdots, G_{N-1})$, then
\enmt
\item $d_{C^0}(G_i, \check{F}_{g^{i}(\theta)}) < 2\epsilon$ for any $0 \leq i \leq N-1$,
\item $G_{N-1}\cdots G_{0}(y) = y + N \rho$ for any $y \in \R$,
\item if $(\check{F}^N)_{\theta}(y) = y + N \rho$ for any $y \in \R$, then $G_i = \check{F}_{g^{i}(\theta)}$ for any $0 \leq i \leq N-1$.
\eenmt

\end{prop}

\begin{proof}
We let $N_0 = N_0(F, \frac{1}{3}\kappa_1)$ be given by Lemma \ref{lem mean motin o(N)}; let $\kappa_1= \kappa_1(F, \epsilon), N_1 = N_1(F,\epsilon)$ be given by Lemma \ref{lem promoting displacement}; let $\kappa_2 = \kappa_2(F,\epsilon)$, $N_2 = N_2(F,\epsilon)$ be given by Lemma \ref{lem local pert}. We let $\kappa_3 = \kappa_3(F,\epsilon) \in (0, \min(\kappa_2, \epsilon))$ be a sufficiently small constant  depending only on $F$ and $\epsilon$, such that for any $n > N_0$, for any $g-$forced map $f'$ with a lift $F'$ satisfying  $d_{C^0}(F',F) < \kappa_3$,  we have
\ary \label{choose kappa 3}
B_n :=\sup_{\theta \in X, y \in \R} |(F'^{n})_{\theta}(y) - y - n \rho| < \frac{2}{3}n\kappa_1.
\eary
The existence of $\kappa_3$ is guaranteed by the compactness of $X$ and the formula
$B_{p+q} \leq B_p + B_q$ for any integers $p,q \geq1$.

We let 
\ary \label{term 100}
N'_2 = N_0 + N_2, \quad p = \lceil 10\epsilon^{-1}\rceil \quad \mbox{and} \quad N_3 = N_1 +10(p-1)N'_2 +1.
\eary
For each $1 \leq i \leq p$, we set $y_i := \frac{i-1}{p}$.
Take an arbitrary $N > N_3$, set $z_i := y_i + N\rho$ for each $1 \leq i \leq p$. Given $\theta \in X$, we inductively define $G_0,\cdots, G_{N-1}$  as follows.

We let $w_1 = (\check{F}^{(p-1)N'_2+1})^{-1}_{g^{N-(p-1)N'_2-1}(\theta)}(z_1)$. 
By $N'_2 > N_0$, \eqref{choose kappa 3} and the hypothesis that $d_{C^0}(\check{F}, F) < \kappa_3$, we have
\ary \label{term 101}
|w_1 + ((p-1)N'_2 + 1) \rho - z_1| <  \frac{2}{3}((p-1)N'_2 + 1) \kappa_1 .
\eary
By \eqref{term 100}, $N-(p-1)N'_2-1 > N_0$. Then by Lemma \ref{lem mean motin o(N)}, we have
\ary \label{term 102} \\ \nonumber
|(F^{N-(p-1)N'_2-1})_{\theta}(y_1)-y_1-(N-(p-1)N'_2-1)\rho| < \frac{1}{3}(N-(p-1)N'_2-1)\kappa_1.
\eary
By \eqref{term 100}, \eqref{term 101}, \eqref{term 102} and $z_1 = y_1 + N\rho$, we have
\aryst
&&|w_1  - (F^{N-(p-1)N'_2-1})_{\theta}(y_1) | \\
 &<&  \frac{1}{3}(N-(p-1)N'_2-1)\kappa_1 +  \frac{2}{3}((p-1)N'_2 +1 ) \kappa_1< \kappa_1  (N- (p-1)N'_2 -1 ).
\earyst
In particular, we have $(\theta, y_1, w_1) \in \Gamma_{N - (p-1)N'_2 - 1}(F, \kappa_1)$.
By \eqref{term 100}, we have $N - (p-1)N'_2-1 > N_1$. We also have $d_{C^0}(\check{F}, F) < \epsilon$. Then we  can apply Lemma \ref{lem cancel mm} to define
\aryst
(G_0, \cdots, G_{N-(p-1)N'_2-2}) := \Phi^{\check{F}}_{N - (p-1)N'_2-1}(\theta, y_1, w_1).
\earyst
By Lemma \ref{lem cancel mm}, we have $G_{N-(p-1)N'_2-2}\cdots G_0(y_1) = w_1$. 

Assume that for some $k \in \{ 2, \cdots, p\}$, $G_0, \cdots, G_{N-(p-k +1)N'_2 - 2}$ are given so that
\ary \label{term 105}\\
\nonumber
(\check{F}^{(p-k+1)N'_2+1})_{g^{N-(p-k+1)N'_2-1}(\theta)}G_{N-(p-k+1)N'_2-2}\cdots G_{0}(y_{l}) = z_{l}, \quad \forall 1 \leq l \leq k-1.
\eary
This is the case when $k=2$ by our construction above. Let
\ary \label{term 200}
 w_k = (\check{F}^{(p-k)N'_2 +1})^{-1}_{g^{N-(p-k)N'_2-1}(\theta)}(z_k).
\eary
We let
\aryst
\underline{y} =  G_{N-(p-k +1)N'_2 - 2} \cdots G_0(y_{k-1}), &&\quad
\overline{y} =  G_{N-(p-k +1)N'_2 - 2} \cdots G_0(y_{1}+1), \\
y' =  G_{N-(p-k +1)N'_2 - 2} \cdots G_0(y_{k}), &&\quad \theta' = g^{N-(p-k+1)N'_2-1}(\theta).
\earyst
By $y_{k-1} < y_k < y_1 +1$, we have $\underline{y} < y' < \overline{y}$. Moreover, by \eqref{term 105} and \eqref{term 200} we have 
\aryst 
&&(\check{F}^{N'_2})_{\theta'}(\underline{y}) = (\check{F}^{(p-k)N'_2 +1})^{-1}_{g^{N-(p-k)N'_2-1}(\theta)}(z_{k-1}) \\
&<& w_k = (\check{F}^{(p-k)N'_2 +1})^{-1}_{g^{N-(p-k)N'_2-1}(\theta)}(z_k) \\
&<& (\check{F}^{N'_2})_{\theta'}(\overline{y}) = (\check{F}^{(p-k)N'_2 +1})^{-1}_{g^{N-(p-k)N'_2-1}(\theta)}(z_{1}+1).
\earyst
Then by $d_{C^0}(\check{F}, F) < \kappa_3 < \kappa_2$ and $N'_2 > N_2$, we can apply Lemma \ref{lem local pert}  to define 
\aryst
(G_{N-(p-k+1)N'_2-1}, \cdots, G_{N-(p-k)N'_2-2}) :=\Psi^{\check{F}}_{0,N'_2}(\theta',  \underline{y}, y', \overline{y}, w_k) .
\earyst
By Lemma \ref{lem local pert}, we have
\aryst
G_{N-(p-k)N'_2-2} \cdots G_{N-(p-k+1)N'_2-1}(y') &=& w_k.
\earyst
Moreover, by (2),(3) in Lemma \ref{lem local pert} and \eqref{term 200}, we verify \eqref{term 105} for $k+1$ in place of $k$. This recovers the induction hypothesis for $k+1$.
We complete the definition of $G_i$ for all $0 \leq i \leq N-2$ when $k=p+1$. Then we have
\ary \label{term 103}
\check{F}_{g^{N-1}(\theta)} G_{N-2} \cdots G_0(y_i) = z_i, \quad \forall i=1,\cdots, p.
\eary
Define
\aryst
H = (\check{F}_{g^{N-1}(\theta)} G_{N-2} \cdots G_0)^{-1} + N\rho \quad \mbox{and} \quad G_{N-1} = H \check{F}_{g^{N-1}(\theta)}.
\earyst
Then by \eqref{term 103}, we have $H(z_i) = z_i$ for any $1 \leq i \leq p$. Thus $d_{C^0}(H,\Id) \leq \frac{1}{p} < \epsilon$. Define 
$\Psi^{\check{F}}_{N}(\theta)=(G_0,\cdots, G_{N-1})$.
 It is direct to see (1)-(3) by our construction.
\end{proof}

\section{Proof of the main results}\label{Proof of the main results}
\begin{proof}[Proof of Theorem \ref{thm a ternary version}:]
Recall that $F$ is a lift of  $f$. In the course of the proof, we abbreviate $\rho(F)$ as $\rho$.
Let integer $d > 0$ be given by Lemma \ref{lem choose returning set}.
We inductively define positive constants $0 <  \varepsilon_{-1} < \varepsilon_{0} < \cdots < \varepsilon_{d}$  by the following formula:
\aryst
\varepsilon_{d} = \frac{\varepsilon}{4(d+2)}\quad \mbox{and} \quad \varepsilon_{d-k} =\frac{1}{2(d+2)}\kappa_3(F, \varepsilon_{d-k+1}) < \varepsilon_{d-k+1}, \forall 1 \leq k \leq d+1.
\earyst
Then we have
\ary \label{term 220} \hspace{1cm}
\quad 2(\varepsilon_{-1} + \cdots + \varepsilon_{d}) < \varepsilon, \quad  2(\varepsilon_{-1} + \cdots + \varepsilon_{k}) < \kappa_3(F, \varepsilon_{k+1}),  \forall -1 \leq k < d.
\eary
We define
\aryst
n_0 = \max_{-1 \leq i \leq d} N_3(F, \varepsilon_{i})+1.
\earyst
By Lemma \ref{lem choose returning set}, we can choose $K$, a compact subset of $X$, that is  $n_0-$good, $D-$spanning and $d-$mild for some $D > 0$.

Let $\{K^i\}_{i=-1}^d, \{Z^i\}_{i=0}^d$ be defined using $K$ as in Section  \ref{sect Stratification}. We will define a sequence of $g-$forced map $f^{(i)}$ for every $-1 \leq i \leq d$ by induction.

We set $f^{(-1)}: = f$ and set $F^{(-1)} := F$. Assume that we have defined $f^{(k)}$ with a lift $F^{(k)}$ for some $-1 \leq k \leq d-1$ such that
\enmt
\item[{\sc(H1)}] $d_{C^0}(F^{(k)}, F) < 2(\varepsilon_{-1} + \cdots + \varepsilon_{k})$, and
\item[{\sc (H2)}] for any $\theta \in K^{k}$, we have $[(F^{(k)})^{\ell(\theta)}]_{\theta}(y) = y + \ell(\theta) \rho$ for every $y \in \R$.
\eenmt
These properties hold for $k=-1$.
For each $-1 \leq j \leq d$, we set
\ary \label{def wj}
W^{j} = \bigcup_{\theta \in K^{j}} \bigcup_{0 \leq j < \ell(\theta)} \{ g^{i}(\theta)\}.
\eary
By Lemma \ref{lemma fact}, we have $\ell(\theta) \leq D$ for every $\theta \in K$, and $W^{d} = X$. We also have the following.
\begin{lemma}\label{lem wj is closed}
Given an integer $0 \leq j \leq d$, let $\{ \theta_n \}_{n \geq 0 }$ be a sequence of points in $K^j$ converging to $\theta'$, and let $ \{\ell_n \in [0, \ell(\theta_n)\}_{n \geq 0}$ be a sequence of integers converging to $\ell'$.
% such that $\{ g^{\ell_n}(\theta_n) \}_{n \geq 0}$ is a Cauchy sequence. 
Then after passing to a subsequence, we have exactly one of the following possibilities:
\enmt
\item[\quad (1)] $\theta' \in Z^{j}$ and $0 \leq \ell' < \ell(\theta')$,
\item[or (2)]  $\theta' \in K^{j-1}$, and there exist a unique $\theta'' \in K^{j-1}$ and a unique $0 \leq \ell'' < \ell(\theta'')$ such that $g^{\ell'}(\theta') = g^{\ell''}(\theta'') \in W^{j-1}$.
\eenmt
In particular, $W^j$ is closed.
\end{lemma}
\begin{proof}
 By our hypothesis, the limit of the sequence $g^{\ell_n}(\theta_n)$ is $g^{\ell'}(\theta')$. By Lemma \ref{lemma fact}(4), we have $\theta' \in K^{j}$.

We first assume that $\theta' \in Z^j$. In this case, by Lemma \ref{lemma fact}(3), we have $\ell(\theta_n) = \ell(\theta')$ for all sufficiently large $n$. Then for all sufficiently large $n$, we have $\ell' = \ell_n < \ell(\theta_n) = \ell(\theta')$. In particular, $g^{\ell'}(\theta') \in W_j$.

Now assume that $\theta' \in K^{j-1}$. We claim that $g^{i}(\theta') \notin int(K)$ for any $1 \leq i \leq \ell'$. Indeed, otherwise there would exist $1 \leq i \leq \ell'$ such that $g^{i}(\theta_n) \in int(K)$ for all sufficiently large $n$. But this is a contradiction, since $\ell^{+}(\theta_n) \geq \ell(\theta_n) > \ell_n= \ell'$ for sufficiently large $n$. Thus our claim is true. In particular, our claim implies that $\ell^{+}(\theta') > \ell'$. Let $k$ be the largest integer in $\{0,\cdots, \ell'\}$ such that $g^{k}(\theta') \in K$ (such $k$ exists since $\theta' \in K$). Then by Lemma \ref{lemma fact}(5), we have $g^{k}(\theta') \in K^{j-1}$ and $0 \leq \ell'-k < \ell(g^{k}(\theta'))$. We let $\theta'' = g^{k}(\theta')$ and $\ell''= \ell'-k$. It is direct to see that $(\theta'', \ell'')$ is the unique pair which fulfills (2).This finishes the proof.
\end{proof}

 In the following, we will define a homeomorphism $\tilde{F} : W^{k+1} \times \R \to g(W^{k+1}) \times \R$ of the form $\tilde{F}(\theta,y) = (g(\theta), \tilde{F}_{\theta}(y))$ such that\clb $\tilde{F}_{\theta} \in \widetilde{Homeo}$ for every $\theta \in W^{k+1}$, and $\tilde{F}_{\theta'} = F^{(k)}_{\theta'}$ for every $\theta' \in W^{k}$.

We define
\ary \label{def tilde f on w k}
\tilde{F}_{\theta'} =  F^{(k)}_{\theta'}, \quad \forall \theta' \in W^{k}.
\eary
By hypothesis {\sc(H1)} and \eqref{term 220}, $d_{C^0}(F^{(k)}, F)  < \kappa_3(F, \varepsilon_{k+1})$. By the definitions of $n_0$ and $K$, we have $\ell(\theta) > N_3(F, \varepsilon_{k+1})$ for any $\theta \in Z^{k+1}$. By Proposition \ref{prop concatenation}, we  define
\ary\label{def tilde f on z k+1}
(\tilde{F}_{\theta}, \cdots, \tilde{F}_{g^{\ell(\theta)-1}(\theta)}) := \Psi^{F^{(k)}}_{\ell(\theta)}(\theta), \quad \forall \theta \in Z^{k+1}.
\eary
By Proposition \ref{prop concatenation}, we have
\ary \label{term 230}
d_{C^0}(\tilde{F}_{g^{i}(\theta)}, F^{(k)}_{g^{i}(\theta)}) &<& 2\varepsilon_{k+1}, \quad \forall \theta \in Z^{k+1}, 0 \leq i < \ell(\theta), \\  \label{term 240}
(\tilde{F}^{\ell(\theta)})_{\theta}(y) &=& y + \ell(\theta) \rho, \quad \forall \theta \in Z^{k+1}, y \in \R.
\eary

We have the following.
\begin{lemma}\label{lem continuation to W k+1}
The map $\tilde{F}$ is continuous.
\end{lemma}
\begin{proof}
It is enough to show that for any $\{\theta_n\}, \{\ell_n\}, \theta', \ell'$ in Lemma \ref{lem wj is closed} with $j=k+1$, we have
\ary \label{tildeftotildef}
\tilde{F}_{g^{\ell_n}(\theta_n)} \to \tilde{F}_{g^{\ell'}(\theta')}\quad\mbox{in  \ $\widetilde{Homeo}$ \ as }\quad n \to \infty.
\eary
We first assume that conclusion (1) in Lemma \ref{lem wj is closed} is true, that is, $\theta' \in Z^{k+1}$. Then \eqref{tildeftotildef} follows immediately from Lemma \ref{lemma fact}(3) and the continuity of $\Psi^{F^{(k)}}_{\ell(\theta')}$.

Now assume that conclusion (2) in Lemma \ref{lem wj is closed} is true. In particular, $\theta' \in K^k$ and $g^{\ell'}(\theta') \in W^k$. It is enough to prove \eqref{tildeftotildef} for two cases: (1) $\theta_n \in K^{k}$ for all $n$; or (2) $\theta_n \in Z^{k+1}$ for all $n$.

 In the first case, we have $g^{\ell_n}(\theta_n) \in W^{k}$ for all $n$. By Lemma \ref{lem wj is closed}, we have $g^{\ell'}(\theta') \in W^{k}$. Then \eqref{tildeftotildef} follows from \eqref{def tilde f on w k} and the fact that $F^{(k)}$ is continuous.
Now assume that the second case is true, namely, $\theta_n \in Z^{k+1}$ for all $n$.  By Lemma \ref{lemma fact}(1), $\ell(\theta_n) \leq D$ for all $n$. After passing to a subsequence, we can assume that there exists an integer $\ell_0 \leq D$ such that $\ell(\theta_n) = \ell_0$ for all $n$. By \eqref{def tilde f on z k+1}, we have
\aryst
\tilde{F}_{g^{\ell'}(\theta_n)} = 
 \mbox{ the $\ell'-$th coordinate of $\Psi^{F^{(k)}}_{\ell_0}(\theta_n)$ }.
\earyst
Then by the continuity of $\Psi^{F^{(k)}}_{\ell_0}$ we have
\aryst
\tilde{F}_{g^{\ell_n}(\theta_n)} \to  \mbox{ the $\ell'-$th coordinate of $\Psi^{F^{(k)}}_{\ell_0}(\theta')$  \quad as }  n \to \infty.
\earyst
It is then enough to show that  the $\ell'-$th coordinate of $\Psi^{F^{(k)}}_{\ell_0}(\theta')$ is $F^{(k)}_{g^{\ell'}(\theta')}$.
By Proposition \ref{prop concatenation}(3), it is enough to verify that $[(F^{(k)})^{\ell_0}]_{\theta'}(y) = y + \ell_0 \rho$ for every $y \in \R$. This last statement follows from {\sc(H2)} and Lemma \ref{lemma fact}(5).  Indeed, by Lemma \ref{lemma fact}(5) we can express $[(F^{(k)})^{\ell_0}]_{\theta'}$ as a composition of maps of the form $[(F^{(k)})^{\ell(\theta'')}]_{\theta''}$ where $\theta'' \in K^k$; by {\sc(H2)}, $[(F^{(k)})^{\ell(\theta'')}]_{\theta''}(y) = y +\ell(\theta'')\rho$ for every $\theta'' \in K^k$ and $y \in \R$.
\end{proof}
We need the following lemma, which is proved in the Appendix.
\begin{prop}[Tietze's Extension Theorem for $\widetilde{Homeo}$]\label{prop-tietze}
Let $F : X \times \R \to X \times \R$ be a lift of a $g-$forced map, and let $M$ be a  compact subset of $X$. Let $G: M \times \R \to g(M) \times \R$ be a continuous map of the form $G(\theta,y) = (g(\theta), G_{\theta}(y))$ such that $G_{\theta} \in \widetilde{Homeo}$ for every $\theta \in M$, and $d_{C^0}(F|_{M \times \R},G)  < c$. Then there exists $G'$, a lift of a $g-$forced map, such that $G'|_{M \times \R} =G|_{M \times \R}$ and $d_{C^0}(F,G') < c$. 
\end{prop}
By \eqref{term 230} and Proposition \ref{prop-tietze}, we can choose $F^{(k+1)} : X \times \R \to X \times \R$, a lift of a $g-$forced map $f^{(k+1)}$, so that:  $F^{(k+1)}_{\theta} = \tilde{F}_{\theta}$ for any $\theta \in W^{k+1}$; and $d_{C^0}(F^{(k+1)}, F^{(k)}) < 2 \varepsilon_{k+1}$. It is straightforward to verify {\sc(H1)},{\sc(H2)} for $k+1$. By induction, {\sc(H1)} and {\sc(H2)} hold for $k=d$

We let $f' = f^{(d)}$ and $F' = F^{(d)}$. By \eqref{term 220} and {\sc(H1)}, $d_{C^0}(f',f) < \varepsilon$.
For every $\alpha \in \R / \Z$, let $T_{\alpha}$ denote the translation $y \mapsto y + \alpha$ on $\R$.
There exists a  map $H: X \times \R \to X \times \R$  of the form $H(\theta,y)= (\theta, H_{\theta}(y))$ and
\aryst
H_{g^{k}(\theta)} = (F'^{k})_{\theta}T_{-k \rho},\quad \forall \theta \in K, 0 \leq k < \ell(\theta).
\earyst
In particular, for each $\theta \in K$, we have $H_{\theta} = \Id$. 
By definition, it is clear that for each $\theta \in X$, $H_{\theta} \in \widetilde{Homeo}$.
By {\sc(H2)} for $k=d$, we have 
\ary \label{term 211}
(F'^{\ell(\theta)})_{\theta}T_{-\ell(\theta) \rho}= \Id, \quad \forall \theta \in K.
\eary
We can verify that $H$ is a homeomorphism by induction. Here we only give an outline since the proof strongly resembles that of  Lemma \ref{lem continuation to W k+1}. We inductively show that the restriction of $H$ to $W^k \times \R$ is continuous for $k=-1,\cdots, d$ ($W^k$ is given by \eqref{def wj}). To pass from stage $k$ to stage $k+1$, 
 it is enough to show that
\ary \label{term 210}
H_{g^{\ell_n}(\theta_n)} \to H_{g^{\ell'}(\theta')}\quad \mbox{as} \quad n \to \infty
\eary 
for any $\{\theta_n\}$,$\{\ell_n\}$, $\theta'$, $\ell'$ given in Lemma \ref{lem wj is closed} for $j=k+1$.  We again divide the proof into two cases, corresponding to $\theta' \in Z^{k+1}$ and $\theta' \in K^k$. In the first case, \eqref{term 210} follows from the continuity of $F'$ and Lemma \ref{lemma fact}(3). In the second case, \eqref{term 210} follows from the induction hypothesis if  $\theta_n \in K^{k}$ for all $n$.  Otherwise we can assume that $\theta_n \in Z^{k+1}$ for all $n$. Then as in the proof of Lemma \ref{lem continuation to W k+1}, we verify \eqref{term 210} using Lemma \ref{lemma fact}(5), the continuity of $F'$ and \eqref{term 211}. 

Let $h$ be the factor of $H$ on $X \times \T$.
It is direct to verify that $f' h(x,y) = h(g(x), y+ \rho)$ for all $(x,y) \in X \times \T$. 
\end{proof}

\begin{proof}[Proof of Theorem \ref{thm2 a ternary version}:]
By Corollary \ref{cor ml is open}, we only need to show the  density of $\cal{ML}$.
Given an arbitrary $g-$forced map $f$ that is not mode-locked. Let $F$ be a lift of $f$.
By Theorem \ref{thm a ternary version},  for any $\varepsilon > 0$, there exists a $g-$forced map $f'$, and a homeomorphism $h : X \times \T \to X \times \T$ of the form $h(\theta,x) = (\theta, h_{\theta}(x))$ with a lift $H: X \times \R \to X \times \R$, such that $d_{C^0}(f,f') < \frac{1}{2}\varepsilon$ and $f' = hRh^{-1}$, where $R$ is  defined as
\aryst
R(x,y) = (g(x), y+\rho(F)), \quad \forall (x,y) \in X \times \T.
\earyst
We let $Q$ be a mode-locked circle homeomorphism that is sufficiently close to rotation $y \mapsto y + \rho(F)$, such that $f'' :=h(g \times Q)h^{-1}$ satisfies that $d_{C^0}(f'', f') < \frac{1}{2}\varepsilon$. Then $d_{C^0}(f,f'') < \varepsilon$. We can verify  that $g \times Q \in \cal{ML}$ by definition. Thus by Corollary \ref{cor ml is open}, $f'' \in \cal{ML}$. This concludes the proof.
\end{proof}

\section{Appendix}
\begin{proof}[Proof of Proposition \ref{prop-tietze}:]
Let $\zeta : \R_{\geq 0}^2 \to \R_{\geq 0}$ be a continuous function such that: 
 $\zeta(r,r') > 0$ for any $r' \in (0, 2r)$; and $\zeta(r,r') =0$ for any $r' \in [2r, \infty)$.

We first show the following weaker version of Proposition \ref{prop-tietze}.
\begin{lemma}\label{lem-tietze}
Proposition \ref{prop-tietze}  is true assuming in addition that:
\enmt
\item $\overline{int(M)} = M$; and
\item for any $\theta \in X \setminus M$, any $\theta' \in M$ with $d(\theta, \theta')<2d(\theta, M)$, we have $d_{C^0}(G_{\theta'}, F_{\theta}) < c$.
\eenmt
\end{lemma}
\begin{proof}
Recall that $\mu$ is an $f-$invariant measure. Since $f$ is minimal, $\mu$ is of full support.
We set $G'(\theta, y) = (g(\theta), G'_{\theta}(y))$ with $G'_{\theta}(y)$ defined as follows. 
For any $\theta \in M$, we set $G'_{\theta} = G_{\theta}$; and for any $\theta \in X \setminus M$, we set
\aryst
G'_{\theta}(y) := (\int_{M} G_{\theta'}(y)\zeta(d(\theta, M), d(\theta, \theta')) d\mu(\theta'))/(\int_M \zeta(d(\theta, M), d(\theta, \theta')) d\mu(\theta')).
\earyst
The above definition makes sense since by (1) and the definition of $\zeta$, for any $\theta \in X \setminus M$, there exists $\theta' \in int(M)$ with $d(\theta, \theta')  < 2 d(\theta, M)$, and as  a consequence $\zeta(d(\theta, M), d(\theta, \theta''))  > 0$ for all $\theta'' \in M$ sufficiently close to $\theta'$. It is direct to verify that $G'_{\theta} \in \widetilde{Homeo}$ for every  $\theta \in X$, and $\theta \mapsto G'_{\theta}$ is continuous over $X$. Thus $G'$ is a homeomorphism. By (2), we have that $d_{C^0}(F,G') < c$. This ends the proof.
\end{proof}

Given an integer $n \geq 2$, we use Tietze's Extension Theorem to define functions $\varphi_{n,i} \in C^0( X,   \R)$,$i=0,\cdots, n-1$, which extend functions $(\theta \mapsto G(\theta, i/n)) \in C^0(M,\R)$,$i=0,\cdots, n-1$ respectively. We set $\varphi_{n,n} = \varphi_{n,0}+1$.  For each $n \geq 2$, we define  $G_n$, a lift of a $g-$forced map, as follows. We set $G_{n}(\theta, y) =(g(\theta), (G_n)_{\theta}(y))$, where for every $ \theta \in X$, every $i\in \{0,\cdots, n-1\}$, every $y \in (i/n, (i+1)/n] $, we set
\aryst
(G_n)_{\theta}(y+k) := ((ny-i)\varphi_{n,i}(\theta) + (i+1-ny)\varphi_{n,i+1}(\theta)) +k, \forall k \in \Z.
\earyst
In another words, $G_n$ is obtained by piecewise affine interpolations using functions $\{\varphi_{n,i}\}_{i=0}^{n}$. Hence $(G_n)_{\theta} \in \widetilde{Homeo}$ for every $\theta \in X$.
Let $\{U_n\}_{n \geq 2}$ be a sequence of open neighbourhoods of $M$ such that $\cap_{n \geq 2} U_n = M$. Moreover, we assume that $U_n \neq X$ and $\overline{U_{n+1}} \subset U_n$ for every $n$.  By Urysohn's lemma, for each $n \geq 2$, there exists $\psi_n \in C^0( X , [0,1])$ such that $\psi_n|_{X \setminus U_n} \equiv 1$ and $\psi_n|_{\overline{U_{n+1}}} \equiv 0$. We set $\psi_1 \equiv 0$.

We define a map $G''$ as follows. For every $\theta \in X$, $y \in \R$, we set $G''(\theta, y) =(g(\theta), G''_{\theta}(y))$ where $G''_{\theta} = G_{\theta}$ if $\theta \in M$; and
\aryst
G''_{\theta}(y) = \frac{\int \sum_{n \geq 1} (\psi_{n+1}-\psi_{n})(\theta') \zeta(d(\theta, M), d(\theta, \theta')) (G_n)_{\theta'}(y) d\mu(\theta')}{\int \sum_{n \geq 1}(\psi_{n+1}-\psi_{n})(\theta') \zeta(d(\theta, M), d(\theta, \theta')) d\mu(\theta')}
\earyst
if $\theta \in X \setminus M$. It is direct to see that $G''_{\theta} \in \widetilde{Homeo}$ for any $\theta \in X$. We claim that: $G''$ is continuous. Indeed, for each $\theta' \in X$, there exist one or two integers $n \geq 1$ such that $(\psi_{n+1}-\psi_n)(\theta') \neq 0$. Thus $G''$ is continuous on $X \setminus M$.
Moreover, for any $n > 0$, there exists $\tau > 0$ such that inequalities $0 < d(\theta, M)< \tau$ implies that any $\theta'$ with  $\zeta(d(\theta, M), d(\theta,\theta')) \neq 0$ satisfies that $\theta' \in B(\theta, 2\tau) \subset B(M, 3\tau) \subset U_{n+1}$; for any such $\theta'$, we have $(\psi_{m+1} - \psi_{m})(\theta') = 0$ for any $m < n$. Then we can  verifies our claim by noting  there exists an integer $n > 0$ such that $d_{C^0}(G_m|_{U_n \times \R}, F|_{U_n \times \R}) < c$ for every $m \geq n$.  Thus $G''$ is a lift of a $g-$forced map whose restriction to $M \times \R$ equals $G$.

Finally, we let $U' \subset U$ be two small open neighborhoods of $M$ with $\overline{U'} \subset U$. We  set $C= \overline{U'} \cup (X \setminus U)$. We  define $G''' : C \times \R \to g(C) \times \R$ by $G'''|_{\overline{U'} \times \R} = G''|_{\overline{U'} \times \R}$ and $G'''|_{(X \setminus U) \times \R} = F |_{(X \setminus U) \times \R}$. It is clear that $\overline{int(C)} = C$. By our hypothesis that $d_{C^0}(G|_{M \times \R}, F|_{M \times \R}) < c$, and by letting $U$ be sufficiently small, we can ensure that condition Lemma \ref{lem-tietze}(2) holds for $(G''',F, C)$ in place of $(G,F,M)$.
We obtain $G'$ as the extension given by  Lemma \ref{lem-tietze} for $(G''', F, C)$ in place of $(G,F,M)$.
%We define a sequence of maps $G'_n$ as follows.
%Assume that $G'_n$ is defined. We define $G''_n$ by formula: $G''_n|_{X \setminus U_n} = G'_{n}|_{X \setminus U_n}$; and $G''_n|_{\overline{U_{n+1}}} = G_{n+1}|_{\overline{U_{n+1}}}$. Then we define $G'_{n+1}$ as the extension given by  Lemma \ref{lem-tietze} for $(G''_n, F, \overline{U_{n+1}} \cup (X \setminus U_n))$ in place of $(G,F, M)$.
%We now show that the sequence $F_n$ converges in the uniform norm.
%We define $G$ be the limit of the sequence $F_n$.
\end{proof}

\end{document}